\documentclass[12pt]{article}
\usepackage[left=2.5cm,right=2.5cm,top=2.5cm,bottom=2.5cm,a4paper]{geometry}

\usepackage[a4paper]{geometry}
\usepackage{graphicx}
\usepackage{microtype}
\usepackage{siunitx}
\usepackage{booktabs}
\usepackage{graphics}
\usepackage{graphicx}
\usepackage{epsfig}
\usepackage{amsmath,amsfonts,amssymb,amsthm}
\usepackage{cleveref}
\usepackage{listings}
\usepackage{paralist}
\usepackage{sectsty}
\usepackage{datetime}
\usepackage[X2,T1]{fontenc} 

\usepackage{authblk}

\title{
	A study on some approximations on the average number of the LLL bases in higher dimensions
}
\author[1]{Jaewon Jung}
\author[2]{Kyunghwan Song\thanks{khsong@jejunu.ac.kr}}
\affil[1]{Department of Mathematics, Korea University, Seoul, Republic of Korea}
\affil[2]{Department of Mathematics, Jeju National University, 102 Jejudaehakro Jeju, 63243, Republic of Korea}

\date{}

\DeclareTextSymbolDefault{\CYRABHDZE}{X2}
\DeclareTextSymbolDefault{\cyrabhdze}{X2}

\numberwithin{equation}{section}
\subsectionfont{\normalfont}

\newtheorem{theorem}{Theorem}[section]
\newtheorem{definition}[theorem]{Definition}
\newtheorem{lemma}[theorem]{Lemma}
\newtheorem{corollary}[theorem]{Corollary}

\newtheorem{problem}[theorem]{problem}
\newtheorem{notation}[theorem]{Notation}

\newcommand{\be}{\begin{equation}}
\newcommand{\ee}{\end{equation}}
\newcommand{\bee}{\begin{equation*}}
\newcommand{\eee}{\end{equation*}}

\usepackage{color}

\begin{document}
	
	\maketitle

\begin{abstract}
	There is a result related to the average number of the $(\delta, \eta)$-LLL bases in dimension $n$ in theoretical sense but the formula seems to be complicated and computing in high dimension takes a long time. In practical sense, we suggest some approximations which can be computed by just storing some constants and computing relatively simple exponential functions.
\end{abstract}



\section{Introduction}
Lattice-based cryptography is a cryptographic system that are based on the hardness of lattice based problems, which is firstly introduced by Ajtai \cite{Ajtai1996} and it is is a promising post-quantum cryptography family. Surely, it has the role of classical cryptography scheme, for example, key exchange and digital signature \cite{Nejatollahi2017}. Furthermore it has various promising applications such as IoT \cite{Khalid2019} and Medical data anlytics\cite{Kocabas2016}. We need some mathematical backgrounds related to linear algebra to know the process of Lattice-based cryptography. Firstly, we introduce the definition of the span of a subset of a vector space.
\begin{definition}
	Let $S = \{\mathbf{v}_1, \mathbf{v}_2, \ldots, \mathbf{v}_n\}$ be a subset of a vector space $V$. Then the span of $S$ is the set of all linear combinations of the vectors in $S$ and denoted by span$(S) = \left\{\sum_{i=1}^n a_i \mathbf{v}_i : a_1, \ldots, a_n \in \mathbb{R}\right\}$.
\end{definition}
Lattice is a set of points, which is called lattice points in $n$-dimensional space. In general, we can say the definition of lattice using linearly independent vectors.
\begin{definition}
	Let $S = \{\mathbf{v}_1, \ldots, \mathbf{v}_n\} \in \mathbb{R}^m$ be a set of linearly independent vectors. The lattice $L$ generated by $S$ is the set of linear combinations of $\mathbf{v}_1, \ldots, \mathbf{v}_n$ with coefficients in $\mathbb{Z}$. That is,
	$$
	L = \left\{\sum_{i=1}^n a_i \mathbf{v}_i : a_1, \ldots, a_n \in \mathbb{Z} \right\}.
	$$
	In this case, we say that $\{\mathbf{v}_1, \ldots, \mathbf{v}_n\}$ is a basis for $L$ and the dimension of $L$ is $n$.
\end{definition}
Related to the definition of lattice, there is one of the well-known problem which is called the Shortest Vector Problem (SVP). If someone finds a sufficiently short vector in a lattice, the vector is a strong candidate of the private key in a lattice-based cryptography \cite{Hoffstein2008}.
\begin{problem}
	The Shortest Vector Problem(SVP): Find a shortest nonzero vector in a lattice $L$, i.e., find a nonzero vector $\mathbf{v} \in L$ that minimizes the Euclidean norm $||\mathbf{v}||$.
\end{problem}
This paper is organized as follows: In Section 2, the definition and properties of LLL-reduced bases is presented. And then we introduce a theoretical result of the average number of the LLL bases in fixed dimension $n$. In Section 3, we analyze the bounds of the average number of the LLL-bases in sufficiently high dimension $n$. And in the last section, we give two approximations of the average number of the LLL-bases in sufficiently high dimension $n$.
\section{LLL-reduced basis}

\begin{notation}
For convenience, we use the following notation
$$
\text{span}_{\mathbb{Z}}(S) = \left\{\sum_{i=1}^n a_i \mathbf{v}_i : a_1, \ldots, a_n \in \mathbb{Z} \right\}
$$ for $S = \{\mathbf{v}_1, \mathbf{v}_2, \ldots, \mathbf{v}_n\}$.
\end{notation}

Note that the shortest vector is not unique in general. Let the dimension of $L$ be $n$. Then the number of the shortest vectors is at least $2^n$. If $a_1 \mathbf{v}_1 + a_2\mathbf{v}_2 + \cdots + a_n \mathbf{v}_n$ is a shortest vector in $L$, then all of the vectors of the form $\pm a_1 \mathbf{v}_1 + \pm a_2 \mathbf{v}_2 + \cdots + \pm a_n \mathbf{v}_n$ are the shortest vectors in $L$. Because of this, for each dimension $n$, Table 3.1 in \cite{Kim2015} focus on the value divided by $2^n$ of the average number of LLL bases. To construct the definition of LLL basis, we have to see the definition of the component, which is related to the span of preceding independent vectors.
\begin{definition}(\cite{Kim2015})
	Let $\{\mathbf{v}_1, \ldots, \mathbf{v}_n\}$ be a basis for $\mathbb{R}^n$. $\mathbf{v}^{*}_i$ is the component of $\mathbf{v}_i$ that is orthogonal to span$(\mathbf{v}_1,\ldots, \mathbf{v}_n)$, and $\mu_{i,j} = <\mathbf{v}_i, \frac{\mathbf{v}^{*}_j}{||\mathbf{v}^{*}_j||^2}>$, $\mathbf{v}^{*}_1 = \mathbf{v}_1$ and we obtain $\mathbf{v}^{*}_i$ for $i = 2,\ldots, n$ using Gram-Schmidt Orthogonalization as follows:
	$$
	\mathbf{v}^{*}_i = \mathbf{v}_i - \sum_{j=1}^{i-1} \mu_{i,j} \mathbf{v}^{*}_j.
	$$
\end{definition} 
Now, we are ready to see the definition of the $(\delta, \eta)$-LLL basis.
\begin{definition} (\cite{Kim2015})
	Let $\{\mathbf{v}_1, \ldots, \mathbf{v}_n\}$ be a basis for $\mathbb{R}^n$, and take two parameters $\frac{1}{2} < \delta < 1$, $\frac{1}{2} < \eta < \delta$. In practice, one often takes $\delta$ and $\eta$ arbitrarily close to, but not equal to, $1$ and $\frac{1}{2}$ respectively. Then a basis $\{\mathbf{v}_1, \ldots, \mathbf{v}_n\}$ is a $(\delta, \eta)$-LLL basis if
	\begin{enumerate}
		\item $|\mu_{i,j}| \leq \eta$ for all $j < i$.
		\item $\delta ||\mathbf{v}^{*}_i|| \leq ||\mathbf{v}^{*}_{i+1} + \mu_{i+1,i} \mathbf{v}^{*}_i||$ for all $i = 1,\ldots, n-1$ (the Lov$\acute{a}$sz condition).
	\end{enumerate}
\end{definition}
In \cite{Kim2015}, the average number of the $(\delta, \eta)$-LLL bases is evaluated and this is directly related to the probability to find a shortest vector in a fixed lattice whose dimension is $n$.
\begin{theorem}(\cite{Kim2015})
	The average number of the $(\delta, \eta)$-LLL bases in dimension $n$ is
	\begin{equation}\label{eq_Kim}
	2\cdot (2\eta)^{\frac{(n-1)(n-2)}{2}} \prod_{i=2}^{n} \frac{S_i (1)}{\zeta(i)}\cdot \frac{1}{n} \prod_{i=1}^{n-1} \frac{1}{i(n-i)}\cdot \prod_{i=1}^{n-1} \int_{-\eta}^{\eta} \sqrt{\delta^2 - x^2}^{-i(n-i)} dx
	\end{equation}
	where $\zeta(i)$ is the Riemann-zeta function and $S_i(x)$ is the surface area of a sphere in $\mathbb{R}^i$ of radius $x$.
\end{theorem}
We give some approximations of (\ref{eq_Kim}) using relatively simple formula in the next section.

\section{The bound of the probability to find a shortest vector using LLL-reduction algorithm}

In fact, Equation (\ref{eq_Kim}) is equal to
	\begin{equation}\label{eq_total}
	2^{\frac{n^2-3n+4}{2}} \eta^{\frac{(n-1)(n-2)}{2}} \prod_{i=2}^{n} \frac{1}{\xi(i)}\cdot \prod_{i=1}^{n-1} \int_{-\eta}^{\eta} \sqrt{\delta^2 - x^2}^{-i(n-i)} dx
\end{equation}
where $\xi(i) = \frac{1}{2}i(i-1)\pi^{-i/2}\Gamma\left(\frac{i}{2}\right)\zeta(i)$, which is called the Riemann-Xi function. Therefore, finding upper and lower bound of the product of the Riemann-Xi functions and that of the product of the integration part of the equation (\ref{eq_total}) are critical parts of finding the approximation of the average number of the LLL bases in sufficiently large dimension $n$. We treat these two parts in the following sections and then summarize them to find upper and lower bound of (\ref{eq_total}).

\subsection{The bound of the product of the Riemann-Xi functions} \label{sec:Riemann-Xi}
Because the Riemann-Xi function is defined by $\xi(s) = \frac{1}{2}s(s-1)\pi^{-s/2}\Gamma\left(\frac{s}{2}\right)\zeta(s)$, firstly we have to find the bounds of $\Gamma(\frac{s}{2})$ and $\zeta(s)$. Immediately, we have
$$
\sqrt{\pi}\left(\frac{s-2}{2e}\right)^{(s-2)/2}(s-2)^{1/2} < \Gamma(\frac{s}{2}) < \sqrt{\pi}\left(\frac{s-2}{2e}\right)^{(s-2)/2}(s-1)^{1/2}
$$
using Lemma 1.9 in \cite{Batir2008}. Also, we can check that
$$
1 < \zeta(s) < 1 + \frac{1}{s-1}
$$
using the integral test. Therefore we get
\begin{align}
	& \frac{1}{2}s(s-1)\pi^{-s/2}  \sqrt{\pi}\left(\frac{s-2}{2e}\right)^{(s-2)/2}(s-2)^{1/2} \label{xi_left} \\
	& < \xi(s) \nonumber \\
	& < \frac{1}{2}s(s-1)\pi^{-s/2} \sqrt{\pi}\left(\frac{s-2}{2e}\right)^{(s-2)/2}(s-1)^{1/2} \left(1 + \frac{1}{s-1}\right). \label{xi_right}
\end{align}
The left side of this inequality (\ref{xi_left}) is greater than
\begin{align}
& \frac{1}{2}(s-2)^2\pi^{-s/2}  \sqrt{\pi}\left(\frac{s-2}{2e}\right)^{(s-2)/2}(s-2)^{1/2}  \nonumber\\
& = \frac{1}{2\sqrt{\pi}}(s-2)^{(s+3)/2}\cdot \left(\frac{1}{2\pi e}\right)^{(s-2)/2} \label{xi_left_v2}
\end{align}
and the right side of this inequality (\ref{xi_right}) is less than
\begin{align}
& \frac{1}{2}(s-1)^2\pi^{-(s-1)/2} \left( \frac{s-1}{2e} \right)^{(s-2)/2} (s-1)^{1/2} \nonumber \\
& = \frac{1}{2\sqrt{\pi}}(s-1)^{(s+3)/2}\cdot \left(\frac{1}{2\pi e}\right)^{(s-2)/2} \label{xi_right_v2}
\end{align}
for $s \geq 6$. From (\ref{xi_left_v2}) and (\ref{xi_right_v2}), we have
\begin{align*}
	& \prod_{s=6}^{n} \left(2\sqrt{\pi}(s-1)^{-(s+3)/2}\cdot \left(2\pi e\right)^{(s-2)/2}\right)  = \left(2\sqrt{\pi}\right)^{(n-5)}(2\pi e)^{(n^2 - 3n - 10)/4} \cdot \prod_{s=6}^{n} (s-1)^{-(s+3)/2} \\
	& < \prod_{s=6}^{n} \frac{1}{\xi(s)} \\ 
	& < \prod_{s=6}^{n} \left(2\sqrt{\pi} (s-2)^{-(s+3)/2}\cdot (2\pi e)^{(s-2)/2}\right) = \left(2\sqrt{\pi}\right)^{(n-5)}(2\pi e)^{(n^2 - 3n - 10)/4} \cdot \prod_{s=6}^{n} (s-2)^{-(s+3)/2}
\end{align*}
and we can simplify the product part as follows:
\begin{align*}
	& \prod_{s=6}^{n} (s-1)^{-(s+3)/2} \\
	& = \exp\left( \sum_{s=6}^n \left(-\frac{s+3}{2}\right)\ln (s-1) \right) \\
	& = \exp\left(-\frac{1}{2}\sum_{s=6}^n (s-1)\ln(s-1) - 2\sum_{s=6}^n \ln(s-1) \right) \\
	& > \exp\left( -\frac{1}{2} \int_{6}^{n+1} (s-1)\ln(s-1) ds -2 \int_{6}^{n+1} \ln(s-1) ds\right) \\ 
	& = \exp\left(-\frac{1}{2}\left(\frac{1}{2}n^2\ln n - \frac{1}{4}n^2 - \frac{25}{2} \ln 5 + \frac{25}{4} \right) - 2\left( n\ln n - n - 5\ln 5 + 5 \right)  \right) \\
	& = \exp\left(-\frac{n(n+8)}{4} \ln n + \frac{n(n+16)}{8} + \frac{25}{4} \ln 5 -\frac{25}{8} + 10 \ln 5 - 10 \right) \\
	& > \exp\left(-\frac{n(n+8)}{4} \ln n + \frac{n(n+16)}{8} + 13.0284  \right),
\end{align*}
\begin{align*}
	& \prod_{s=6}^{n} (s-2)^{-(s+3)/2} \\
	& = \exp\left( \sum_{s=6}^n \left(-\frac{s+3}{2}\right)\ln (s-2) \right) \\
	& = \exp\left(-\frac{1}{2}\sum_{s=6}^n (s-2)\ln(s-2) -\frac{5}{2} \sum_{s=6}^n \ln(s-2) \right) \\
	& < \exp\left( -\frac{1}{2} \int_{5}^{n} (s-2)\ln(s-2) ds -\frac{5}{2} \int_{5}^{n} \ln(s-2) ds\right) \\ 
	& = \exp\left(-\frac{1}{2}\left(\frac{(n-2)^2}{2}\ln n - \frac{(n-2)^2}{4} - \frac{9}{2} \ln 3 + \frac{9}{4} \right) - \frac{5}{2} \left( (n-2) \ln (n-2) - (n-2) - 3\ln 3 + 3 \right)  \right) \\
	& = \exp\left(-\frac{(n+8)(n-2)}{4} \ln (n-2) + \frac{(n+18)(n-2)}{8} + \frac{9}{4} \ln 3 -\frac{9}{8} + \frac{15}{2} \ln 3 -\frac{15}{2} \right) \\
	& < \exp\left(-\frac{(n+8)(n-2)}{4} \ln (n-2) + \frac{(n+18)(n-2)}{8} + 2.08647  \right),
\end{align*}
Therefore we have
\begin{align*}
	& \prod_{s=2}^{5} \frac{1}{\xi(s)}\cdot \left(2\sqrt{\pi}\right)^{(n-5)}(2\pi e)^{(n^2 - 3n - 10)/4} \cdot  \exp\left(-\frac{n(n+8)}{4} \ln n + \frac{n(n+16)}{8} + 13.0284  \right) \\
	& < \prod_{s=2}^{n} \frac{1}{\xi(s)} \\ 
	& < \prod_{s=2}^{5} \frac{1}{\xi(s)} \cdot \left(2\sqrt{\pi}\right)^{(n-5)}(2\pi e)^{(n^2 - 3n - 10)/4} \\
	& \quad \cdot  \exp\left(-\frac{(n+8)(n-2)}{4} \ln (n-2) + \frac{(n+18)(n-2)}{8} + 2.08647  \right).
\end{align*}
Also, we can verify
\begin{align*}
	& (2\sqrt{\pi})^{n-5} (2\pi e)^{(n^2 - 3n - 10)/4} \cdot 2^{(n^2-3n+4)/2}\cdot \eta^{(n-1)(n-2)/2} \\
	& = \exp\Bigg( n^2\left(\frac{3}{4}\ln 2 + \frac{1}{4}\ln \pi + \frac{1}{2} \ln \eta + \frac{1}{4} \right) \\
	& \quad + n\left(-\frac{5}{4}\ln 2 -\frac{1}{4}\ln \pi -\frac{3}{2}\ln \eta -\frac{3}{4} \right) \\
	& \quad + \left(-\frac{11}{2}\ln 2 -5\ln \pi +\ln \eta - \frac{5}{2} \right)\Bigg).
\end{align*}
Hence we have a simplified form of lower bound
\begin{align*}
	& \prod_{s=2}^{5} \frac{1}{\xi(s)} \cdot \exp\Bigg(-\frac{1}{4}n^2\ln n + n^2 \left(\frac{3}{4}\ln 2 + \frac{1}{4}\ln \pi + \frac{1}{2} \ln \eta + \frac{3}{8}\right) \\
	& \quad -2n\ln n + n\left(-\frac{5}{4}\ln 2 - \frac{1}{4}\ln \pi  -\frac{3}{2} \ln \eta + \frac{5}{4} \right) \\
	& \quad + 0.9924 + \ln \eta \Bigg).
\end{align*}
of $2^{\frac{n^2-3n+4}{2}} \eta^{\frac{(n-1)(n-2)}{2}} \prod_{s=2}^{n} \frac{1}{\xi(s)}$.

Using $\ln n - \ln (n-2) > \frac{2}{n}$, we have
\begin{align*}
	& -\frac{1}{4}n^2 \ln(n-2) + \frac{1}{8}n^2 - \frac{3}{2}n\ln (n-2) + 2n + 4\ln(n-2) - \frac{9}{2} \\
	& < -\frac{1}{4}n^2 \ln n + \frac{1}{8} n^2 - \frac{3}{2}n\ln n + \frac{5}{2} n + 4\ln n - \frac{3}{2}.
\end{align*}
Hence we have an simplified form of upper bound
\begin{align*}
	& \prod_{s=2}^{5} \frac{1}{\xi(s)} \cdot \exp\Bigg(-\frac{1}{4}n^2\ln n + n^2\left(\frac{3}{4}\ln 2 + \frac{1}{4}\ln \pi + \frac{1}{2}\ln \eta + \frac{3}{8}\right) \\
	& \quad -\frac{3}{2}n\ln n + n\left(-\frac{5}{4}\ln 2 - \frac{1}{4}\ln \pi -\frac{3}{2}\ln \eta + \frac{7}{4}\right) \\
	& \quad + 4\ln n - 11.4495 + \ln \eta \Bigg)
\end{align*}
of $2^{\frac{n^2-3n+4}{2}} \eta^{\frac{(n-1)(n-2)}{2}} \prod_{s=2}^{n} \frac{1}{\xi(s)}$.

Since $\ln(\prod_{s=2}^{5} \frac{1}{\xi(s)}) = 1.85914510535951$, we have 
\begin{lemma}
	\begin{align*}
		& \exp\Bigg(-\frac{1}{4}n^2\ln n + n^2 \left(\frac{3}{4}\ln 2 + \frac{1}{4}\ln \pi + \frac{1}{2} \ln \eta + \frac{3}{8}\right) \\
		& \quad -2n\ln n + n\left(-\frac{5}{4}\ln 2 - \frac{1}{4}\ln \pi  -\frac{3}{2} \ln \eta + \frac{5}{4} \right) \\
		& \quad + 2.8515 + \ln \eta \Bigg) \\
		& < 2^{\frac{n^2-3n+4}{2}} \eta^{\frac{(n-1)(n-2)}{2}} \prod_{s=2}^{n} \frac{1}{\xi(s)} \\
		& < \exp\Bigg(-\frac{1}{4}n^2\ln n + n^2\left(\frac{3}{4}\ln 2 + \frac{1}{4}\ln \pi + \frac{1}{2}\ln \eta + \frac{3}{8}\right) \\
		& \quad -\frac{3}{2}n\ln n + n\left(-\frac{5}{4}\ln 2 - \frac{1}{4}\ln \pi -\frac{3}{2}\ln \eta + \frac{7}{4}\right) \\
		& \quad + 4\ln n - 9.5903 + \ln \eta \Bigg).
	\end{align*}
\end{lemma}

\subsection{The bound of the product of the integration part} \label{sec:Integration}
Let $n \geq 22$. We start with the following theorem that is used in the main computation.
\begin{theorem}\label{Thm_exp}
	Let $0 < x < 1$ be a real number. Then we have
	\begin{equation}\label{eq_prod_1-x^k}
		\prod_{k=1}^{n} (1-x^k) \geq \exp\left( \frac{x(1-x^n)}{1-x}\left(\ln(1-x) - 1\right) \right).
	\end{equation}
\end{theorem}
\begin{proof}
	Using Taylor series for $\ln(1-x^2)$ at $x = 0$, we have
	\begin{align*}
		-\sum_{k=1}^{n}\ln(1-x^k) = \sum_{k=1}^{n}\left(\sum_{i=1}^{\infty} \frac{x^{ki}}{i}\right) = \sum_{i=1}^{\infty}\frac{1}{i}\left(\sum_{k=1}^{n}(x^i)^k\right).
	\end{align*}
	Note that it is equal to
	\begin{align*}
		\sum_{i=1}^{\infty} \frac{1}{i}\frac{x^i\left(1-x^{in}\right)}{1-x^i} = \frac{x(1-x^n)}{1-x}\left(1 + \sum_{i=2}^{\infty} \frac{x^{i-1}(1 + x^n + \cdots + x^{(i-1)n})}{i\left(1 + x + \cdots + x^{i-1}\right)}\right),
	\end{align*}
and it is smaller than
	\begin{equation*}
		\frac{x(1 - x^n)}{1-x} \left( 1+\sum_{i=2}^{\infty} \frac{x^{i-1}}{i-1} \right) = \frac{x(1 - x^n)}{1-x} \left( 1- \ln (1-x) \right).
	\end{equation*}
It leads to the following inequality
	\begin{equation*}		
		\sum_{k=1}^{n} \ln (1 - x^k) >  \frac{x(1 - x^n)}{1-x} \left( \ln (1-x) -1 \right) 
	\end{equation*}
and therefore we have	
	\begin{equation*}
		\prod_{k=1}^{n} (1 - x^k) = \exp \left( \sum_{k=1}^{n} \ln (1 - x^k) \right) > \exp \left( \frac{x(1 - x^n)}{1-x} \left( \ln (1-x) -1 \right) \right).
	\end{equation*}
\end{proof}
Then let us apply the change of variables to the integration part as follows:
\begin{align*}
& \prod_{i=1}^{n-1} \int_{-\eta}^{\eta} \sqrt{\delta^2 - x^2}^{-i(n-i)} dx \\
& = \prod_{i=1}^{n-1} \int_{-\sin^{-1}\left(\frac{\eta}{\delta}\right)}^{\sin^{-1}\left(\frac{\eta}{\delta}\right)} (\delta \cos \theta)^{-i(n-i)} \delta \cos \theta d\theta (x = \delta \sin \theta) \\
& = \prod_{i=1}^{n-1} \delta^{-i(n-i) + 1}\int_{-\sin^{-1}\left(\frac{\eta}{\delta}\right)}^{\sin^{-1}\left(\frac{\eta}{\delta}\right)} ( \cos \theta)^{-i(n-i) + 1} d\theta \\
& = \prod_{i=1}^{n-1} 2\delta^{-i(n-i) + 1}\int_{0}^{\sin^{-1}\left(\frac{\eta}{\delta}\right)} ( \sec \theta)^{i(n-i) - 1} d\theta.
\end{align*}
For convenience, let $m = i(n-i) - 1$. Since
\begin{align*}
	& \int  ( \sec \theta)^{m} d\theta \\
	& = \frac{\sec^{m-2}x \tan x}{m-1} + \frac{m-2}{m-1} \int  ( \sec \theta)^{m-2} d\theta \\
	& = \frac{\sec^{m-2}x \tan x}{m-1} + \frac{m-2}{m-1}\frac{\sec^{m-4}x \tan x}{m-3} + \frac{m-2}{m-1}\frac{m-4}{m-3} \int  ( \sec \theta)^{m-4} d\theta \\
	& = \frac{\sec^{m-2}x \tan x}{m-1} + \frac{m-2}{m-1}\frac{\sec^{m-4}x \tan x}{m-3} +  \frac{m-2}{m-1}\frac{m-4}{m-3} \frac{\sec^{m-6}x \tan x}{m-5} \\
	& \quad + \frac{m-2}{m-1}\frac{m-4}{m-3}\frac{m-6}{m-5} \int  ( \sec \theta)^{m-6} d\theta \\
	& = \cdots,
\end{align*}
we have a lower bound of $\int_{0}^{\sin^{-1}\left(\frac{\eta}{\delta}\right)} ( \sec \theta)^{i(n-i) - 1} d\theta$ as follows: if $m$ is even, we have
\begin{align*}
	& \int_{0}^{\sin^{-1}\left(\frac{\eta}{\delta}\right)} ( \sec \theta)^{i(n-i) - 1} d\theta \\
	& > \left[\frac{\sec^{m-2}x \tan x}{m-1} + \frac{\sec^{m-4}x \tan x}{m-1} + \frac{\sec^{m-6}x \tan x}{m-1} + \cdots + \frac{\sec^2x \tan x}{m-1} + \frac{\tan x}{m-1}\right]^{\sin^{-1}\left(\frac{\eta}{\delta}\right)}_0 \\
	& = \left[\frac{\sec^{m-2}x \tan x}{m-1}\left(\frac{1-\cos^m x}{1-\cos^2 x}\right)\right]^{\sin^{-1}\left(\frac{\eta}{\delta}\right)}_0 \\
	& = \frac{\sec^{m-2}\left(\sin^{-1}\left(\frac{\eta}{\delta}\right)\right)\tan \left(\sin^{-1}\left(\frac{\eta}{\delta}\right)\right) }{m-1} \frac{1-\cos^m \left(\sin^{-1}\left(\frac{\eta}{\delta}\right)\right)}{1-\cos^2 \left(\sin^{-1}\left(\frac{\eta}{\delta}\right)\right)} \\
	& = \frac{\sec^{m-2}\left(\sin^{-1}\left(\frac{\eta}{\delta}\right)\right)\tan \left(\sin^{-1}\left(\frac{\eta}{\delta}\right)\right) }{m-1} \frac{1-\cos^m \left(\sin^{-1}\left(\frac{\eta}{\delta}\right)\right)}{\frac{\eta^2}{\delta^2}} \\
	& = \frac{\cos^{-m+2}\left(\sin^{-1}\left(\frac{\eta}{\delta}\right)\right)\tan \left(\sin^{-1}\left(\frac{\eta}{\delta}\right)\right) }{m-1} \frac{1-\cos^m \left(\sin^{-1}\left(\frac{\eta}{\delta}\right)\right)}{\frac{\eta^2}{\delta^2}} \\
	& = \frac{\delta^2}{\eta^2(m-1)}\left(\frac{\frac{\eta}{\delta}}{\cos^{m-1}\left(\sin^{-1}\left(\frac{\eta}{\delta}\right)\right)} - \frac{\eta}{\delta}\cos\left(\sin^{-1}\left(\frac{\eta}{\delta}\right) \right) \right) \\
	& = \frac{\delta}{\eta (m-1)}\frac{1}{\cos^{m-1}\left(\sin^{-1}\left(\frac{\eta}{\delta}\right)\right) }\left(1 - \cos^{m}\left(\sin^{-1}\left(\frac{\eta}{\delta}\right)\right) \right) \\
	& = \frac{\delta}{\eta \left(i(n-i) - 2\right)}\frac{1}{\cos^{i(n-i)-2}\left(\sin^{-1}\left(\frac{\eta}{\delta}\right)\right)}\cdot \left(1-\cos^{i(n-i)-1}\left(\sin^{-1}\left(\frac{\eta}{\delta}\right)\right)\right).
\end{align*}
Therefore
\begin{align*}
	& \prod_{i=1}^{n-1} \int_{0}^{\sin^{-1}\left(\frac{\eta}{\delta}\right)} ( \sec \theta)^{i(n-i) - 1} d\theta \\
	& > \prod_{i=1}^{n-1}  \frac{\delta}{\eta \left(i(n-i) - 2\right)}\frac{1}{\cos^{i(n-i)-2}\left(\sin^{-1}\left(\frac{\eta}{\delta}\right)\right)}\cdot \left(1-\cos^{i(n-i)-1}\left(\sin^{-1}\left(\frac{\eta}{\delta}\right)\right)\right) \\
	& = \frac{\delta^{n-1}}{\eta^{n-1}}\cdot \cos^{-\frac{(n-1)(n-3)(n+4)}{6}}\left(\sin^{-1}\left(\frac{\eta}{\delta}\right)\right)\cdot \prod_{i=1}^{n-1} \frac{1-\cos^{i(n-i)-1}\left(\sin^{-1}\left(\frac{\eta}{\delta}\right)\right)}{i(n-i)-2} \\
	& > \frac{\delta^{n-1}}{\eta^{n-1}}\cdot \cos^{-\frac{(n-1)(n-3)(n+4)}{6}}\left(\sin^{-1}\left(\frac{\eta}{\delta}\right)\right)\cdot \prod_{i=1}^{n-1} \frac{1-\cos^{i(n-i)-1}\left(\sin^{-1}\left(\frac{\eta}{\delta}\right)\right)}{i(n-i)-1} \\
	& > \frac{\delta^{n-1}}{\eta^{n-1}}\cdot \cos^{-\frac{(n-1)(n-3)(n+4)}{6}}\left(\sin^{-1}\left(\frac{\eta}{\delta}\right)\right)\cdot \prod_{i=1}^{n-1} \frac{1-\cos^{in}\left(\sin^{-1}\left(\frac{\eta}{\delta}\right)\right)}{in} \\
	& > \frac{\delta^{n-1}}{\eta^{n-1}}\cdot \frac{1}{(n-1)!~n^{n-1}}\cdot \cos^{-\frac{(n-1)(n-3)(n+4)}{6}}\left(\sin^{-1}\left(\frac{\eta}{\delta}\right)\right) \\
	& \quad \cdot \exp\left(\frac{\cos^{n}\left(\sin^{-1}\left(\frac{\eta}{\delta}\right)\right)\left(1-\cos^{(n-1)n}\left(\sin^{-1}\left(\frac{\eta}{\delta}\right)\right)\right)}{1-\cos^{n}\left(\sin^{-1}\left(\frac{\eta}{\delta}\right)\right)}\left(\ln\left(1-\cos^{n}\left(\sin^{-1}\left(\frac{\eta}{\delta}\right)\right) \right) - 1 \right)\right)\\
	& \text{by Theorem \ref{Thm_exp}.}
\end{align*}

To find a lower bound when $m$ is odd, we have to set the following lemma.
\begin{lemma}\label{lem_even}
	Let $0 < x < \frac{\sqrt{3}}{2}$ and $l \geq 19$. Then we have
	\begin{equation}\label{eq_lem_1}
		\frac{1-x^{l-2}}{l} > \frac{1-x^{l+2}}{l+2}.
	\end{equation}
\end{lemma}
\begin{proof}
	Note that Eq. (\ref{eq_lem_1}) is equivalent to
	$$
	(l+2)x^{l-2} - lx^{l+2} < 2
	$$
	whose LHS is lower than $f(l) = (l+2)\left(\frac{\sqrt{3}}{2}\right)^{l-2}$. Since $f'(l) < 0$ for any $l \geq 6$ and $f(19) < 2$, we conclude that the inequality (\ref{eq_lem_1}) is satisfied when $l \geq 19$.
\end{proof}
Then we are ready to find a lower bound of the integration part when $m$ is odd. We start with the following inequalities and equations.
\begin{align*}
	& \int_{0}^{\sin^{-1}\left(\frac{\eta}{\delta}\right)} ( \sec \theta)^{i(n-i) - 1} d\theta \\
	& > \bigg[\frac{\sec^{m-2}x \tan x}{m-1} + \frac{\sec^{m-4}x \tan x}{m-1} + \frac{\sec^{m-6}x \tan x}{m-1} + \cdots + \frac{\sec^3x \tan x}{m-1} \\
	& \quad + \frac{1}{m-1}\ln|\sec x + \tan x|\bigg]^{\sin^{-1}\left(\frac{\eta}{\delta}\right)}_0 \\
	& = \left[\frac{\sec^{m-2}x \tan x}{m-1}\left(\frac{1-\cos^{m-3} x}{1-\cos^2 x}\right) + \frac{1}{m-1}\ln|\sec x + \tan x|\right]^{\sin^{-1}\left(\frac{\eta}{\delta}\right)}_0 \\
	& = \frac{\delta}{\eta (m-1)}\frac{1}{\cos^{m-1}\left(\sin^{-1}\left(\frac{\eta}{\delta}\right)\right) }\left(1 - \cos^{m-3}\left(\sin^{-1}\left(\frac{\eta}{\delta}\right)\right) \right) \\
	& \quad + \frac{1}{m-1}\ln\left|\sec \left(\sin^{-1}\left(\frac{\eta}{\delta}\right)\right) + \tan \left(\sin^{-1}\left(\frac{\eta}{\delta}\right)\right)\right| \\
	& > \frac{\delta}{\eta (m-1)}\frac{1}{\cos^{m-1}\left(\sin^{-1}\left(\frac{\eta}{\delta}\right)\right) }\left(1 - \cos^{m-3}\left(\sin^{-1}\left(\frac{\eta}{\delta}\right)\right) \right).
\end{align*}

Similar to the case when $m$ is even, we have a lower bound
\begin{align*}
	& \prod_{i=1}^{n-1} \int_{0}^{\sin^{-1}\left(\frac{\eta}{\delta}\right)} ( \sec \theta)^{i(n-i) - 1} d\theta \\
	& > \prod_{i=1}^{n-1}  \frac{\delta}{\eta \left(i(n-i) - 2\right)}\frac{1}{\cos^{i(n-i)-2}\left(\sin^{-1}\left(\frac{\eta}{\delta}\right)\right)}\cdot \left(1-\cos^{i(n-i)-4}\left(\sin^{-1}\left(\frac{\eta}{\delta}\right)\right)\right) \\
	& = \frac{\delta^{n-1}}{\eta^{n-1}}\cdot \cos^{-\frac{(n-1)(n-3)(n+4)}{6}}\left(\sin^{-1}\left(\frac{\eta}{\delta}\right)\right)\cdot \prod_{i=1}^{n-1} \frac{1-\cos^{i(n-i)-4}\left(\sin^{-1}\left(\frac{\eta}{\delta}\right)\right)}{i(n-i)-2} \\
	& > \frac{\delta^{n-1}}{\eta^{n-1}}\cdot \cos^{-\frac{(n-1)(n-3)(n+4)}{6}}\left(\sin^{-1}\left(\frac{\eta}{\delta}\right)\right)\cdot \prod_{i=1}^{n-1} \frac{1-\cos^{i(n-i)}\left(\sin^{-1}\left(\frac{\eta}{\delta}\right)\right)}{i(n-i)} \text{ by Lemma \ref{lem_even}.} \\
	& > \frac{\delta^{n-1}}{\eta^{n-1}}\cdot \cos^{-\frac{(n-1)(n-3)(n+4)}{6}}\left(\sin^{-1}\left(\frac{\eta}{\delta}\right)\right)\cdot \prod_{i=1}^{n-1} \frac{1-\cos^{in}\left(\sin^{-1}\left(\frac{\eta}{\delta}\right)\right)}{in} \\
	& > \frac{\delta^{n-1}}{\eta^{n-1}}\cdot \frac{1}{(n-1)!~n^{n-1}}\cdot \cos^{-\frac{(n-1)(n-3)(n+4)}{6}}\left(\sin^{-1}\left(\frac{\eta}{\delta}\right)\right) \\
	& \quad \cdot \exp\left(\frac{\cos^{n}\left(\sin^{-1}\left(\frac{\eta}{\delta}\right)\right)\left(1-\cos^{(n-1)n}\left(\sin^{-1}\left(\frac{\eta}{\delta}\right)\right)\right)}{1-\cos^{n}\left(\sin^{-1}\left(\frac{\eta}{\delta}\right)\right)}\left(\ln\left(1-\cos^{n}\left(\sin^{-1}\left(\frac{\eta}{\delta}\right)\right) \right) - 1 \right)\right) \\
	& \text{by Theorem \ref{Thm_exp}.}
\end{align*}
Note that a formula for lower bound is equal for any $m$. Therefore we have the following
\begin{lemma}\label{lem_int_lower}
	Let $n \geq 22$. Then we have
	\begin{align*}
	& \prod_{i=1}^{n-1} \int_{0}^{\sin^{-1}\left(\frac{\eta}{\delta}\right)} ( \sec \theta)^{i(n-i) - 1} d\theta \\
	& > \frac{\delta^{n-1}}{\eta^{n-1}}\cdot \frac{1}{(n-1)!~n^{n-1}}\cdot \cos^{-\frac{(n-1)(n-3)(n+4)}{6}}\left(\sin^{-1}\left(\frac{\eta}{\delta}\right)\right) \\
	& \quad \cdot \exp\left(\frac{\cos^{n}\left(\sin^{-1}\left(\frac{\eta}{\delta}\right)\right)\left(1-\cos^{(n-1)n}\left(\sin^{-1}\left(\frac{\eta}{\delta}\right)\right)\right)}{1-\cos^{n}\left(\sin^{-1}\left(\frac{\eta}{\delta}\right)\right)}\left(\ln\left(1-\cos^{n}\left(\sin^{-1}\left(\frac{\eta}{\delta}\right)\right) \right) - 1 \right)\right).
	\end{align*}
\end{lemma}

Because $\int  ( \sec \theta)^{m} d\theta < \frac{\sec^{m-2}x \tan x}{m-1} + \frac{m-2}{m-1} \int  ( \sec \theta)^{m} d\theta$, an upper bound is as follows:
\begin{lemma}\label{lem_int_upper}
	\begin{equation*}
		 \prod_{i=1}^{n-1} \int_{0}^{\sin^{-1}\left(\frac{\eta}{\delta}\right)} ( \sec \theta)^{i(n-i) - 1} d\theta < \frac{\eta^{n-1}}{\delta^{n-1}}\cos^{-\frac{(n-1)(n-3)(n+4)}{6}}\left(\sin^{-1}\left(\left(\frac{\eta}{\delta}\right) \right)\right).
	\end{equation*}
\end{lemma}
\begin{proof}
\begin{align*}
	& \prod_{i=1}^{n-1} \int_{0}^{\sin^{-1}\left(\frac{\eta}{\delta}\right)} ( \sec \theta)^{i(n-i) - 1} d\theta \\
	& < \prod_{i=1}^{n-1} \sec^{m-2}\left(\sin^{-1}\left(\frac{\eta}{\delta}\right)\right) \tan \left(\sin^{-1}\left(\frac{\eta}{\delta}\right)\right) \\
	& = \tan^{n-1} \left(\sin^{-1}\left(\frac{\eta}{\delta}\right)\right) \cos^{-\frac{(n-1)(n^2 + n - 18)}{6}}\left(\sin^{-1}\left(\left(\frac{\eta}{\delta}\right) \right)\right) \\
	& = \frac{\eta^{n-1}}{\delta^{n-1}}\cos^{-\frac{(n-1)(n-3)(n+4)}{6}}\left(\sin^{-1}\left(\left(\frac{\eta}{\delta}\right) \right)\right) 
\end{align*}
\end{proof}
Combining Lemma \ref{lem_int_lower} and \ref{lem_int_upper}, we have the following
\begin{lemma}\label{lem_int_upp_low}
	Let $n \geq 22$ and $t = \cos\left(\sin^{-1}\left(\frac{\eta}{\delta}\right)\right)$. Then we have
		\begin{align*}
			& \frac{\delta^{n-1}}{\eta^{n-1}}\cdot \frac{1}{(n-1)!~n^{n-1}}\cdot t^{-\frac{(n-1)(n-3)(n+4)}{6}} \cdot \exp\left(\frac{t^{n}\left(1-t^{(n-1)n}\right)}{1-t^{n}}\left(\ln\left(1-t^{n}\right) - 1 \right)\right) \\
			& < \prod_{i=1}^{n-1} \int_{0}^{\sin^{-1}\left(\frac{\eta}{\delta}\right)} ( \sec \theta)^{i(n-i) - 1} d\theta \\
			& < \frac{\eta^{n-1}}{\delta^{n-1}}\cdot t^{-\frac{(n-1)(n-3)(n+4)}{6}}.
		\end{align*}
\end{lemma}
Then we have a more simplified lower and upper bound of $\displaystyle \prod_{i=1}^{n-1} \int_{0}^{\sin^{-1}\left(\frac{\eta}{\delta}\right)} ( \sec \theta)^{i(n-i) - 1} d\theta$ as follows:
\begin{lemma}
	Let $n \geq 22$, $t = \cos\left(\sin^{-1}\left(\frac{\eta}{\delta}\right)\right)$ and $a = t^n$. Then we have
	\begin{align*}
		& \frac{\delta^{n-1}}{\eta^{n-1}}\cdot \exp \left( -2(n-1)\ln n + \frac{a(1 - a^{n-1})}{1-a}(\ln (1-a) - 1) + \left(-\frac{n^2}{6} + 3\right)\ln a\right) \\
		& < \prod_{i=1}^{n-1} \int_{0}^{\sin^{-1}\left(\frac{\eta}{\delta}\right)} ( \sec \theta)^{i(n-i) - 1} d\theta \\
		& < \frac{\eta^{n-1}}{\delta^{n-1}} a^{-\frac{n^2}{6}}.
	\end{align*}
\end{lemma}
\begin{proof}
	\begin{align*}
		& \frac{\delta^{n-1}}{\eta^{n-1}}\cdot \frac{1}{(n-1)!~n^{n-1}}\cdot t^{-\frac{(n-1)(n-3)(n+4)}{6}} \cdot \exp\left(\frac{t^{n}\left(1-t^{(n-1)n}\right)}{1-t^{n}}\left(\ln\left(1-t^{n}\right) - 1 \right)\right) \\
		& > \frac{\delta^{n-1}}{\eta^{n-1}}\frac{1}{(n-1)!n^{n-1}}a^{-\frac{n^2}{6} + 3} \exp\left(\frac{a(1 - a^{n-1})}{1-a}(\ln (1-a) - 1)\right) \\
		& > \frac{\delta^{n-1}}{\eta^{n-1}}\cdot \exp \left( -2(n-1)\ln n + \frac{a(1 - a^{n-1})}{1-a}(\ln (1-a) - 1) + \left(-\frac{n^2}{6} + 3\right)\ln a\right).
	\end{align*}
Also we have
$$
		\frac{\eta^{n-1}}{\delta^{n-1}} t^{-\frac{(n-1)(n-3)(n+4)}{6}} < \frac{\eta^{n-1}}{\delta^{n-1}} a^{-\frac{n^2}{6}}.
$$
\end{proof}
\section{Conclusion}
Finally, we obtain a simplified bound of the average number of the $(\delta, \eta)$-LLL bases in dimension $n$. We summarize and give two approximations in this section.

In the previous sections \ref{sec:Riemann-Xi} and \ref{sec:Integration}, we have the following simplified bound of the average number of the $(\delta,\eta)$-LLL bases in dimension $n$.
\begin{theorem}
	Let $n \geq 22$, $t = \cos\left(\sin^{-1}\left(\frac{\eta}{\delta}\right)\right)$, and $a = t^n$. Then we have
	\begin{align*}
		& \frac{\delta^{n-1}}{\eta^{n-1}}\cdot \exp \Bigg(-\frac{1}{4}n^2\ln n + n^2\bigg(\frac{3}{4}\ln 2 + \frac{1}{4}\ln \pi + \frac{1}{2} \ln \eta + \frac{3}{8} - \frac{1}{6} \ln a\bigg) \\
		& \quad - 4n\ln n + n\bigg(-\frac{5}{4}\ln 2 - \frac{1}{4} \ln \pi - \frac{3}{2}\ln \eta + \frac{5}{4}\bigg) + 2.8515 + \ln \eta \\
		& \quad + \frac{a(1 - a^{n-1})}{1-a}(\ln (1-a) - 1) + 3\ln a\Bigg) \\
		& < 2^{\frac{n^2-3n+4}{2}} \eta^{\frac{(n-1)(n-2)}{2}} \prod_{i=2}^{n} \frac{1}{\xi(i)}\cdot \prod_{i=1}^{n-1} \int_{-\eta}^{\eta} \sqrt{\delta^2 - x^2}^{-i(n-i)} dx \\
		& < \frac{\eta^{n-1}}{\delta^{n-1}} \cdot \exp \Bigg( -\frac{1}{4}n^2\ln n + n^2\left(\frac{3}{4}\ln 2 + \frac{1}{4}\ln \pi + \frac{1}{2}\ln \eta + \frac{3}{8} -\frac{1}{6}\ln a \right) \\
		& \quad -\frac{3}{2}n\ln n + n\left(-\frac{5}{4}\ln 2 - \frac{1}{4}\ln \pi -\frac{3}{2}\ln \eta + \frac{7}{4}\right) \\
		& \quad + 4\ln n - 9.5903 + \ln \eta \Bigg).
	\end{align*}
\end{theorem}
Because we are focusing on the practical case, we assume that $\eta$ and $\delta$ are sufficiently close to $\frac{1}{2}$ and $1$, respectively.

Let $\frac{1}{2}<\eta<\frac{3}{4\sqrt{2}}$, $\frac{3}{4}<\delta<1$ and  $n \geq 22$. Then $\frac{\sqrt{2}}{2} < t < \frac{\sqrt{3}}{2}$, and therefore for any integer $c_1 \leq 0$, we have
\begin{align*}
& \frac{\delta^{n-1}}{\eta^{n-1}}\cdot \exp \Bigg(n\bigg(-\frac{5}{4}\ln 2 - \frac{1}{4} \ln \pi - \frac{3}{2}\ln \eta + \frac{5}{4}\bigg) + 2.8515 + \ln \eta \\
& \quad + \frac{a(1 - a^{n-1})}{1-a}(\ln (1-a) - 1) + 3\ln a\Bigg) \\
& > \exp \Bigg(n\bigg(-\frac{5}{4}\ln 2 - \frac{1}{4} \ln \pi - \frac{3}{2}\ln \frac{3}{4\sqrt{2}} + \frac{5}{4}\bigg) + 2.8515 + \ln \frac{1}{2} \\
& \quad + \frac{(\frac{\sqrt{3}}{2})^{n} (1 - (\frac{\sqrt{2}}{2})^{n(n-1)})}{1-(\frac{\sqrt{3}}{2})^{n}}(\ln (1-(\frac{\sqrt{3}}{2})^{n}) - 1) + 3\ln (\frac{\sqrt{2}}{2})^{n}+(n-1)\ln\sqrt{2}\Bigg) \\
& > c_{1} n \ln n.
\end{align*}
Also, for any integer $c_2 \geq 1$, we have
\begin{align*}
& \frac{\eta^{n-1}}{\delta^{n-1}}\cdot \exp \Bigg(n\left(-\frac{5}{4}\ln 2 - \frac{1}{4}\ln \pi -\frac{3}{2}\ln \eta + \frac{7}{4}\right) + 4\ln n - 9.5903 + \ln \eta \Bigg) \\
& <  \exp \Bigg(n\left(-\frac{5}{4}\ln 2 - \frac{1}{4}\ln \pi -\frac{3}{2}\ln \frac{1}{2} + \frac{7}{4}\right) + 4\ln n - 9.5903 + \ln \frac{3}{4\sqrt{2}} +(n-1) \ln \frac{\sqrt{2}}{2} \Bigg) \\
& < c_{2} n \ln n.
\end{align*}
Combining these inequalities, we have the following
\begin{theorem}
	Let $\frac{1}{2} < \eta < \frac{3}{4\sqrt{2}}, \frac{3}{4} < \delta < 1$, $n \geq 22$, $t = \cos\left(\sin^{-1}\left(\frac{\eta}{\delta}\right)\right)$, and $a = t^n$. Then we have
\begin{align*}
&  \exp \Bigg(-\frac{1}{4}n^2\ln n + n^2\bigg(\frac{3}{4}\ln 2 + \frac{1}{4}\ln \pi + \frac{1}{2} \ln \eta + \frac{3}{8} - \frac{1}{6} \ln a\bigg) - 4n\ln n \Bigg) \\
& < 2^{\frac{n^2-3n+4}{2}} \eta^{\frac{(n-1)(n-2)}{2}} \prod_{i=2}^{n} \frac{1}{\xi(i)}\cdot \prod_{i=1}^{n-1} \int_{-\eta}^{\eta} \sqrt{\delta^2 - x^2}^{-i(n-i)} dx \\
& < \exp \Bigg( -\frac{1}{4}n^2\ln n + n^2\left(\frac{3}{4}\ln 2 + \frac{1}{4}\ln \pi + \frac{1}{2}\ln \eta + \frac{3}{8} -\frac{1}{6}\ln a \right) -\frac{1}{2}n\ln n \Bigg).
\end{align*}
\end{theorem}
$a = t^n$, or equivalently, $\ln a = n\ln t$, implies that $\exp\left(-\frac{1}{6}n^2 \ln a\right) = a^{-\frac{1}{6}n^2} = t^{-\frac{1}{6}n^3}$ is the leading term of the approximation of 
$$
2^{\frac{n^2-3n+4}{2}} \eta^{\frac{(n-1)(n-2)}{2}} \prod_{i=2}^{n} \frac{1}{\xi(i)}\cdot \prod_{i=1}^{n-1} \int_{-\eta}^{\eta} \sqrt{\delta^2 - x^2}^{-i(n-i)} dx.
$$
Therefore we suggest two approximations as follows:
\begin{corollary}
	(A rough version) Let $\frac{1}{2} < \eta < \frac{3}{4\sqrt{2}}, \frac{3}{4} < \delta < 1$, and $n \geq 22$. Then we have a rough approximation 
	$$
	\cos^{-\frac{1}{6}n^3}\left(\sin^{-1}\left(\frac{\eta}{\delta}\right)\right)
	$$
	of the average number of the $(\delta, \eta)$-LLL bases in dimension $n$. Immediately, the following asymptotic behavior
	$$
	\lim_{n\rightarrow \infty}\frac{\ln\left(\cos^{-\frac{1}{6}n^3}\left(\sin^{-1}\left(\frac{\eta}{\delta}\right)\right)\right)}{\ln\left(2 (2\eta)^{\frac{(n-1)(n-2)}{2}} \prod_{i=2}^{n} \frac{S_i (1)}{\zeta(i)}\cdot \frac{1}{n} \prod_{i=1}^{n-1} \frac{1}{i(n-i)}\cdot \prod_{i=1}^{n-1} \int_{-\eta}^{\eta} \sqrt{\delta^2 - x^2}^{-i(n-i)} dx\right)} = 1
	$$
	holds.
\end{corollary}
\begin{corollary}
	(A tight version) Let $\frac{1}{2} < \eta < \frac{3}{4\sqrt{2}}, \frac{3}{4} < \delta < 1$, and $n \geq 22$. Then we have a tight approximation 
	$$
	\cos^{-\frac{1}{6}n^3}\left(\sin^{-1}\left(\frac{\eta}{\delta}\right)\right) \cdot \exp\left( -\frac{1}{4}n^2\ln n + n^2\bigg(\frac{3}{4}\ln 2 + \frac{1}{4}\ln \pi + \frac{1}{2} \ln \eta + \frac{3}{8} \bigg) - cn\ln n  \right)
	$$
	of the average number of the $(\delta, \eta)$-LLL bases in dimension $n$. Note that $c$ is between $\frac{1}{2}$ and $4$.
\end{corollary}

\section*{Acknowledgements}
This work was supported by the research grant of Jeju National University in 2022.


%
%

\end{document}